\documentclass[12pt,a4paper]{amsart}
\usepackage{enumerate,amsmath,amsthm,amssymb,cite,mathrsfs,graphicx,caption,subfigure}
\usepackage{float}

\DeclareMathOperator{\RE}{Re} \DeclareMathOperator{\IM}{Im}

\numberwithin{equation}{section}
\newtheorem{theorem}{Theorem}[section]
\newtheorem{lemma}[theorem]{Lemma}
\newtheorem{corollary}[theorem]{Corollary}

\newtheorem{example}[theorem]{Example}
\theoremstyle{remark}
\newtheorem{remark}[theorem]{Remark}

\allowdisplaybreaks

\makeatletter
\@namedef{subjclassname@2020}{\textup{2020} Mathematics Subject Classification}
\makeatother
\begin{document}
\title[Sufficient conditions for Strong starlikeness]{Sufficient conditions for Strong starlikeness}

\dedicatory{Dedicated to Prof.\ Dato' Indera Rosihan M. Ali}

\author[K. Sharma]{Kanika sharma}
\address{Department of Mathematics, Atma Ram Sanatan Dharma College, University of Delhi, Delhi--110 021, India}
\email{kanika.divika@gmail.com}

\author[N. E. Cho]{Nak Eun Cho}
\address{Department of Applied Mathematics, Pukyong National University, Busan 608--737, South Korea}
\email{necho@pknu.ac.kr}

\author{V. Ravichandran}
\address{Department of Mathematics, National Institute of Technology, Tiruchirappalli--620 015, India} \email{ravic@nitt.edu; vravi68@gmail.com}

\begin{abstract}
Let $p$ be an analytic function defined on the open unit disc $\mathbb{D}$ with $p(0)=1$ and $0< \alpha \leq 1$. The conditions on complex valued functions $C$, $D$ and $E$ are obtained for $p$ to be subordinate to $((1+z)/(1-z))^{\alpha}$ when $C(z) z^{2}p''(z)+D(z)zp'(z) + E(z)p(z)=0$. Sufficient conditions for confluent (Kummer) hypergeometric function and generalized and normalized Bessel function of the first kind of complex order to be subordinate to $((1+z)/(1-z))^{\alpha}$ are obtained as applications. The conditions on $\alpha$ and $\beta$ are derived for $p$ to be subordinate to $((1+z)/(1-z))^{\alpha}$ when $1+\beta zp'(z)/p^{n}(z)$ with $n=1,2$ is subordinate to $1+4z/3+2z^{2}/3=:\varphi_{CAR}(z)$. Similar problems were investigated for $\RE p(z)>0$ when the functions $p(z)+\beta zp'(z)/p^{n}(z)$  with $n=0,2$ is subordinate to $\varphi_{CAR}(z)$. The condition on $\beta$ is determined for $p$ to be subordinate to $((1+z)/(1-z))^{\alpha}$ when $p(z)+\beta zp'(z)/p^{n}(z)$ with $n=0,1,2$ is subordinate to $((1+z)/(1-z))^{\alpha}$.
\end{abstract}

\maketitle
\section{introduction}
For a natural number $n$, let $\mathcal{H}[a,n]$ be the class of all analytic functions $p$ defined on the open unit disc $\mathbb{D}$ of the form $p(z)=a+p_{n}z^{n}+p_{n+1}z^{n+1}+\cdots$. An analytic function $p \in \mathcal{H}[1,1]$ is a function with a positive real part if $\RE p(z)>0$ and  the class of all such functions is denoted by $\mathcal{P}$. Let $\mathcal{A}_n=\{zh : h \in \mathcal{H}[1,n]\}$ and let $\mathcal{A}:=\mathcal{A}_1$. Let $\mathcal{S}$ denote the subclass of $\mathcal{A}$ consisting of univalent (one-to-one) functions. For an analytic function $\varphi$ with $\varphi(0)=1$, let
\[ \mathcal{S}^*(\varphi) := \left\{ f\in\mathcal{A}:\frac{ zf'(z)}{f(z)}\prec\varphi(z)\right\}.\]
This class unifies various classes of starlike functions when $\RE \varphi>0$. Shanmugam \cite{shan} studied the convolution properties of this class when $\varphi$ is convex while Ma and Minda \cite{mamin2} investigated the growth, distortion and coefficient estimates under less restrictive assumption that $\varphi$ is starlike and $\varphi(\mathbb{D})$ is symmetric with respect to the real axis. For $0\leq \alpha <1$, the class $\mathcal{S}^*((1+(1-2\alpha)z)/(1-z))=:\mathcal{S}^*(\alpha)$ is the class of starlike functions of order $\alpha$, introduced by Robertson \cite{rob}. The class $\mathcal{S}^*:=\mathcal{S}^*(0)$ is the familiar class of starlike functions. For $0<\alpha\leq1,$ $\mathcal{S}^*[\alpha]:=\mathcal{S}^*(((1+z)/(1-z))^{\alpha})$ is the class of strongly starlike functions of order $\alpha$. For $0\leq \beta<1$, a function $f \in \mathcal{A}$ is said to be close-to-convex of order $\beta$ \cite{ural,goodman} if $\RE(zf'(z)/g(z))>\beta$ for some $g\in \mathcal{S}^*$ (in general, $g$ need not be normalized but we add this extra normalization). More results regarding these classes can be found in \cite{cho,lee}. Recently, the authors have investigated the sufficient conditions for a function to belong to various subclasses of starlike functions in \cite{sharma1,sharma2,sharma}. The class $S^*_C:=\mathcal{S}^*(\varphi_{CAR})$, where $\varphi_{CAR}(z)=1+4z/3+2z^{2}/3$ was introduced and studied recently in \cite{jain,sharma,sharma2}. Precisely, $f\in S^*_C$ provided $zf'(z)/f(z)$ lies in the region bounded by the cardioid
\[ \Omega_C :=\left\{w=u+iv: (9 u^2+9 v^2-18u+5)^2- 16 (9 u^2+9 v^2- 6u+1)=0\right\}. \]
The   functions
\[f_{1}(z)=\frac{4z}{(2-z)^{2}}\quad\text{and}\quad f_{2}(z)=z\left(1+ \frac{z}{4}\right)^{2}\]
belong to the  class $S^*_C$. The image of the unit disc under the function $q_{i}(z):=zf_{i}'(z)/f_{i}(z),$ $(i=1,2)$ is contained inside the cardioid $\Omega_C$ (see Figure \ref{p3fig1}).

 \begin{figure}[htb!]
 \centering
 {\includegraphics[width=.3\linewidth]{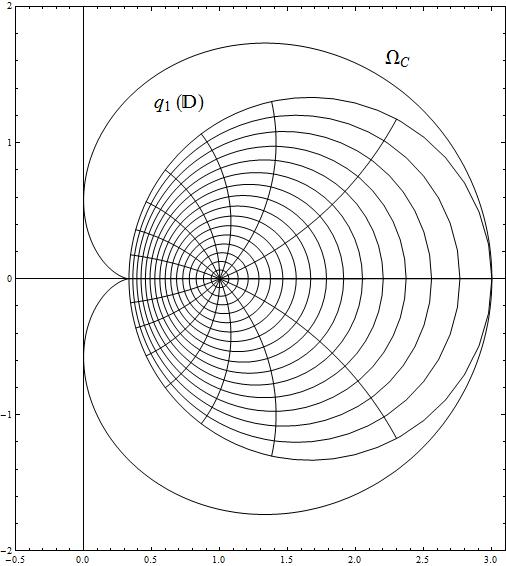}}
 \hspace{2cm}
 {\includegraphics[width=.3\linewidth]{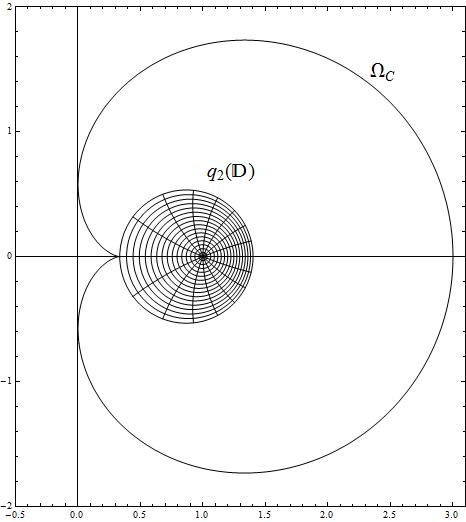}}
 \caption{Graph showing $q_i(\mathbb{D})\subset \Omega_C$ ($i=1,2$).}\label{p3fig1}
 \end{figure}
In Section ~\ref{p6sec2}, we investigate the conditions on complex valued functions $C, D$ and $E$ so that $C(z) z^{2}p''(z)+D(z)zp'(z) + E(z)p(z)=0$ implies that $p(z)\prec ((1+z)/(1-z))^{\alpha}$, $(0<\alpha\leq 1)$. As an application, we obtain sufficient conditions for generalized and normalized Bessel function of the first kind of complex order and confluent (Kummer) hypergeometric function to be subordinate to  $((1+z)/(1-z))^{\alpha}$. These results are the extension of \cite[Theorem 2.2, p.\ 29]{baricz1} for $|\arg(p(z))|<\alpha \pi/2$.  As an application, we also deduce some conditions on the  functions $f \in \mathcal{A}$, $g \in \mathcal{H}[1,1]$ so that their product $fg \in \mathcal{S}^*[\alpha]$. We also obtained the sufficient conditions for $h \in \mathcal{A}_n$ to belong to the class $\mathcal{S}^*[\alpha]$.

A convex function is starlike of order $1/2;$ analytically, $p(z)+z p'(z)/p(z)\prec (1+z)/(1-z)$ implies that $p(z)\prec 1/(1-z)$. Similarly, a sufficient condition for a function $p$ to be a function with positive real part is that $p(z)+zp'(z)/p(z)\prec R(z)$, where $R$ is the open door mapping given by $R(z):=(1+z^2+4z)/(1-z^{2})$.
Several authors have investigated similar results for functions to belong to certain regions in right half plane. For example, Ali \textit{et al.}\cite{ali2} determined the condition on $\beta$ for $p(z)\prec \sqrt{1+z}$ when $1+\beta zp'(z)/p^{n}(z)$ with $n=0,1,2$ or $(1-\beta)p(z)+\beta p^{2}(z)+\beta zp'(z)$ is subordinate to $\sqrt{1+z}$. For related results, see \cite{ali1,ali2,pap,sokol09b,sharma,sharma2,vir}. We investigate a similar problem for regions that were considered recently by many authors. In Section ~\ref{p6sec3}, we determine conditions on $\alpha$ and $\beta$ so that $p(z)\prec ((1+z)/(1-z))^{\alpha}, (0<\alpha\leq1)$ when $1+\beta zp'(z)/p(z)$ or $1+\beta zp'(z)/p^{2}(z)\prec \varphi_{CAR}(z)$. Condition on $\beta$ is also determined so that $p(z)+\beta zp'(z)$ or $p(z)+\beta zp'(z)/p^{2}(z)\prec\varphi_{CAR}(z)$ implies $p(z)\prec(1+z)/(1-z)$. For $0<\alpha\leq1$, we also find conditions on $\alpha$ and $\beta$ so that $p(z)+\beta z p'(z)/p^{k}(z)$, $(k=0,1,2)$ is subordinate to $((1+z)/(1-z))^{\alpha}$ implies $p(z)\prec ((1+z)/(1-z))^{\alpha}$. As an application of our results, we give sufficient conditions for $f\in \mathcal{A}$ to belong to the various subclasses of starlike functions.

The following results are required in our investigation.

\begin{lemma}\cite[Theorem 2.3i, p.35]{ural}\label{p6lem1}
Let $\Omega\subset \mathbb{C}$ and suppose that $\psi: {\mathbb{C}}^{3} \times \mathbb{D} \to \mathbb{C}$ satisfies the condition $\psi(i\rho, \sigma, \mu + i \nu; z) \notin \Omega$ whenever $\rho,$ $\sigma$, $\mu$ and $\nu$ are real numbers, $\sigma \leq -n(1+\rho^2)/2$, $\mu+\sigma \leq 0$. If $p \in \mathcal{H}[1,n]$ and $\psi(p(z),zp'(z),z^{2}p''(z);z)\in \Omega$ for $z \in \mathbb{D}$, then $\RE p(z)>0$ in $\mathbb{D}$.
\end{lemma}

\begin{lemma}\cite[Theorem 3.4i, p.134]{ural}\label{p6int2}
Let $q$ be univalent in $\mathbb{D}$ and let $\varphi$ and $\nu$  be analytic in a domain $D$ containing $q(\mathbb{D})$ with $\varphi(w)\neq 0$ when $w\in q(\mathbb{D})$. Set
$ Q(z):=zq'(z)\varphi(q(z))$, $ h(z):=\nu(q(z))+Q(z)$.
Suppose that (i) either $h$ is convex or $Q(z)$ is starlike univalent in $\mathbb{D}$ and
 (ii) $\RE({zh'(z)}/{Q(z)})>0$ for $z\in\mathbb{D}$.
If $p$ is analytic in $\mathbb{D}$, $p(0)=q(0)$ and satisfies
\begin{equation}\label{p2eq1.1}
\nu(p(z))+zp'(z)\varphi(p(z)) \prec \nu(q(z))+zq'(z)\varphi(q(z)),
\end{equation}
then $p\prec q$ and $q$ is the best dominant.
\end{lemma}

\begin{lemma} \cite[Theorem 3.4a, p.120]{ural}\label{p6int3}
Let $q$ be analytic in $\mathbb{D}$ and $\phi$ be analytic in a domain $D$ containing $q(\mathbb{D})$ and suppose that
(i)  $\RE\phi(q(z))> 0$ and either (ii)  $q$ is convex, or (iii)  $Q(z)=zq'(z)\phi(q(z))$ is starlike.
If $p$ is analytic in $\mathbb{D}$, $p(0)=q(0)$, $p(\mathbb{D})\subset D$ and
$p(z)+zp'(z)\phi(p(z)) \prec q(z)$,
then $p\prec q$.
\end{lemma}

\section{Results associated with strong starlikeness}\label{p6sec2}

In the first result, we derived the conditions on the complex valued functions $C, D$ and $E$ so that $C(z) z^{2}p''(z)+D(z)zp'(z) + E(z)p(z)=0$ implies that $p(z)\prec ((1+z)/(1-z))^{\alpha}$, $(0< \alpha\leq 1)$.

\begin{theorem}\label{p6thm2.1}
Let $n$ be a positive integer, $0 < \alpha \leq 1, C(z)=C \geq 0$. Suppose that the functions $D, E:\mathbb{D}\to\mathbb{C}$ satisfy
\begin{equation}\label{p6eqn2.8}
|\IM E(z)|< n \alpha(\RE D(z)-C).
\end{equation}
If $p \in \mathcal{H}[1,n]$ satisfies the equation
\begin{equation}\label{p6eqn2.1}
C z^{2}p''(z)+D(z)zp'(z) + E(z)p(z) = 0
\end{equation}
and $p(z)\neq 0$, then $p(z)\prec ((1+z)/(1-z))^{\alpha}$.
\end{theorem}

\begin{proof}
For $p \in \mathcal{H}[1,n]$ with $p(z)\neq 0$, define the function $q:\mathbb{D}\to\mathbb{C}$ by
\begin{equation}\label{p6eqn2.50}
q(z)=p^{\frac{1}{\alpha}}(z).
\end{equation}
Then $q$ is analytic in $\mathbb{D}$ and $q(0)=1$. Note that
\begin{equation}\label{p6eqn2.2}
p(z)=q^{\alpha}(z),
\end{equation}
\begin{equation}\label{p6eqn2.3}
p'(z)=\alpha q^{\alpha -1}(z) q'(z)
\end{equation}
and
\begin{equation}\label{p6eqn2.4}
p''(z)=\alpha((\alpha-1)q^{\alpha-2}(z)(q'(z))^{2}+q^{\alpha-1}(z)q''(z)).
\end{equation}
Using ~\eqref{p6eqn2.2}, ~\eqref{p6eqn2.3} and ~\eqref{p6eqn2.4} in ~\eqref{p6eqn2.1}, a simple computation shows that $q$ satisfies the following equation
\begin{equation}\label{p6eqn2.5}
\begin{split}
C\alpha ((\alpha-1)\frac{(z q'(z))^{2}}{q(z)}+ z^{2}q''(z))+ D(z)\alpha zq'(z)+E(z)q(z)=0.
\end{split}
\end{equation}
Let $\psi: {\mathbb{C}}^{3} \times \mathbb{D} \to \mathbb{C}$ be defined by
\begin{equation}\label{p6eqn2.6}
\psi(r,s,t;z)= C\alpha ((\alpha-1)\frac{s^{2}}{r}+ t)+ D(z)\alpha s+E(z)r.
\end{equation}
Then the condition ~\eqref{p6eqn2.5} is same as $\psi(q(z),zq'(z),z^{2}q''(z);z)\in \Omega=\{0\}$. To show that $\RE q(z)>0$ for $z \in \mathbb{D}$, by Lemma ~\ref{p6lem1}, it is sufficient to prove that $\RE \psi(i\rho, \sigma, \mu + i \nu; z)<0$ for $z \in \mathbb{D}$, and for all real $\rho,$ $\sigma$, $\mu$ and $\nu$ satisfying $\sigma \leq -n(1+\rho^2)/2$, $\mu+\sigma \leq 0$. For $z\in \mathbb{D}$, it follows from ~\eqref{p6eqn2.6} that
\begin{equation}\label{p6eqn2.7}
\RE \psi(i\rho, \sigma, \mu + i \nu; z)=C \alpha\mu + \RE D(z) \alpha \sigma- \IM E(z)\rho.\\
\end{equation}
Using conditions $\RE D(z) > C \geq 0$, $\mu+\sigma \leq 0$ and $\sigma \leq -n(1+\rho^2)/2$, we get
\[ C \alpha\mu + \RE D(z)\sigma \alpha \leq -C\alpha\sigma+\RE D(z)\sigma\alpha  \leq -n(1+\rho^2)\alpha(\RE D(z)- C)/2 .\]
Thus from ~\eqref{p6eqn2.7}, we have
\begin{align*}
\RE \psi(i\rho, \sigma, \mu + i \nu; z)& \leq -\frac{n}{2}(1+\rho^2)\alpha(\RE D(z)- C)-\IM E(z) \rho\\
&= -\frac{n}{2}\alpha(\RE D(z)- C)\rho^2-\IM E(z) \rho -\frac{n}{2}\alpha(\RE D(z)- C) \\
&=:a\rho^{2}+b\rho+c,
\end{align*}
where $a=c=-n\alpha(\RE D(z)- C)/2$ and $b=-\IM E(z)$. In view of ~\eqref{p6eqn2.8}, we see that $a<0$ and $b^{2}-4ac < 0$ and hence $a\rho^{2}+b\rho+c <0$ which proves that $\RE \psi(i\rho, \sigma, \mu + i \nu; z)<0$ for all $z \in \mathbb{D}$. Hence by Lemma ~\ref{p6lem1}, we deduce that $\RE q(z)>0$ and by using ~\eqref {p6eqn2.50}, it reduces to
\[ p^{\frac{1}{\alpha}}(z) \prec \frac{1+z}{1-z} \]
which implies that $p(z) \prec ((1+z)/(1-z))^{\alpha}$. \end{proof}

By taking $\alpha=1$ in Theorem ~\ref{p6thm2.1}, we get the following result.

\begin{corollary}\label{p6cor1}
Let $n$ be a positive integer, $C(z)=C \geq 0$. Suppose that the functions $D, E :\mathbb{D}\to\mathbb{C}$ satisfy
\[|\IM E(z)|< n (\RE D(z)-C).\]
If $p \in \mathcal{H}[1,n]$ satisfies $C z^{2}p''(z)+D(z)zp'(z) + E(z)p(z)=0$ and $p(z)\neq 0,$
then $\RE p(z)>0$.
\end{corollary}

By taking $C(z)=0$ in Theorem ~\ref{p6thm2.1}, we get the following result for first order differential subordination.

\begin{corollary}\label{p6cor2}
Let $n$ be a positive integer, $0 < \alpha \leq 1$. Suppose that the functions $D, E:\mathbb{D}\to\mathbb{C}$ satisfy
\[|\IM E(z)|< n \alpha \RE D(z).\]
If $p \in \mathcal{H}[1,n]$ satisfies
$D(z)zp'(z) + E(z)p(z)=0$ and $p(z)\neq 0,$
then $p(z)\prec ((1+z)/(1-z))^{\alpha}$.
\end{corollary}

\begin{remark}\label{p6rmk2}
The Corollary ~\ref{p6cor2} for $\alpha =1$ should be compared with \cite[Corollary 4.1a.1, p.\ 189]{ural}
\end{remark}

The confluent (Kummer) hypergeometric function $\Phi(a,c;z)$ is given by
\begin{align}
\Phi(a,c;z) = \frac{\Gamma{(c)}}{\Gamma{(a)}} \sum_{n=0}^\infty \frac{\Gamma{(a+n)}}{\Gamma{(c+n)}} \frac{z^n}{n!} = \sum_{n=0}^\infty
\frac{(a)_{n}}{(c)_{n}} \frac{z^n}{n!},
\end{align}
where $ a, c \in \mathbb{C}$, $c \neq 0, -1, -2, \cdots$, and $(\lambda)_n$ denotes the Pochhammer symbol defined by $(\lambda)_0=1$, $(\lambda)_n= \lambda(\lambda+1)_{n-1}$.   The function $\Phi \in \mathcal{H}[1,1]$ is a solution of the differential equation
\begin{align}\label{p6eqn2.16}
z\Phi''(a, c; z) + (c-z)\Phi'(a, c; z) - a \Phi(a, c; z) = 0
\end{align}
introduced by Kummer in 1837 \cite{TM}. The function $\Phi(a, c; z)$ satisfies the following recursive relation
\[c\Phi'(a; c; z) = a \Phi(a+1; c+1; z).\]
When $ \RE c > \RE a >0$, the function $\Phi$ can be expressed in the integral form
\begin{align*}
\Phi(a; c; z) =  \frac{ \Gamma{(c)}} { \Gamma{(a)}\Gamma{(c-a)}}  \int_{0}^1 t^{a-1} (1-t)^{c-a-1} e^{ t z} dt.
\end{align*}
There has been several works \cite{AE, miller-dif-sub,Rus-Vsing,saiful} studying geometric properties of the function $\Phi(a; c; z)$, such as on its close-to-convexity, starlikeness and convexity. By the use of Theorem ~\ref{p6thm2.1}, we obtain the following sufficient conditions for $\Phi(a,c;z) \prec ((1+z)/(1-z))^{\alpha}$, ($0 < \alpha \leq 1$).

\begin{corollary}\label{p6cor6}
Let $0 < \alpha \leq 1$. If $a,c \in \mathbb{R}$ satisfy
\begin{equation}\label{p6eqn2.40}
|c-1|>\sqrt{1+a^{2}/\alpha^{2}}
\end{equation}
and $\Phi(a;c;z) \neq 0$, then $\Phi(a;c;z) \prec ((1+z)/(1-z))^{\alpha}$.
\end{corollary}

\begin{proof}
In view of ~\eqref{p6eqn2.16}, the function $\Phi(a,c;z)$ satisfies ~\eqref{p6eqn2.1} with $C(z)=1$, $D(z)=c-z$ and $E(z)=-az$. For $z=x+iy \in \mathbb{D}$, we have
\begin{equation*}
\begin{split}
(\IM E(z))^{2}&-(\RE D(z)-C)^{2}\alpha^{2}\\
&\quad{}=a^{2}y^{2}-(c-x-1)^{2}\alpha^{2}\\
&\quad{}<a^{2}(1-x^{2})-(c-x-1)^{2}\alpha^{2}\\
&\quad{}=-(a^{2}+\alpha^{2})x^{2}+2\alpha^{2}(c-1)x+(a^{2}-(c-1)^{2}\alpha^{2}) =:px^{2}+qx+r,
\end{split}
\end{equation*}
where $p=-(a^{2}+\alpha^{2})$, $q=2\alpha^{2}(c-1)$ and $r=a^{2}-(c-1)^{2}\alpha^{2}$. Clearly, $p<0$ and using ~\eqref{p6eqn2.40}, we get $q^{2}-4pr<0$. So, $(\IM E(z))^{2}< (\RE D(z)-C)^{2}\alpha^{2}$ and thus, all the conditions of the Theorem ~\ref{p6thm2.1} are satisfied. Hence, $\Phi(a;c;z) \prec ((1+z)/(1-z))^{\alpha}$.
\end{proof}

The following example illustrates the Corollary ~\ref{p6cor6}.

\begin{example}
	Clearly, the following functions
	\[\Phi_{1}(z)=\Phi(2;6;z)=\frac{20 \left(z^3+6 z^2+18 z+6 e^z (z-4)+24\right)}{z^5}\]
\begin{equation*}
\begin{split}
	\Phi_{2}(z)=\Phi(5;10;z)=&\frac{15120}{z^9}\Big(-z^4-20 z^3-180 z^2+e^z (z^4-20 z^3+180 z^2-840 z+1680)\\
&-840 z-1680\Big)
	\end{split}
\end{equation*}
satisfy the hypothesis of Corollary ~\ref{p6cor6}. Therefore $\Phi_i(z)\prec ((1+z)/(1-z))^{\alpha_{i}}$ (for $i=1,2)$ with $\alpha_{1}>1/\sqrt{6}$ and $\alpha_{2}>\sqrt{5}/4$. These subordinations are illustrated graphically in Figure \ref{p6fig1}.

\begin{figure}[htb!]
 \centering
 {\includegraphics[width=.3\linewidth]{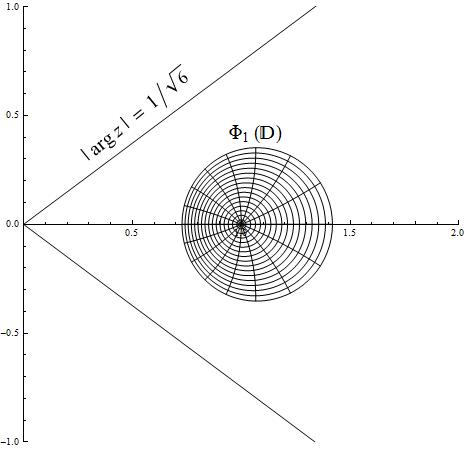}}
 \hspace{2cm}
 {\includegraphics[width=.3\linewidth]{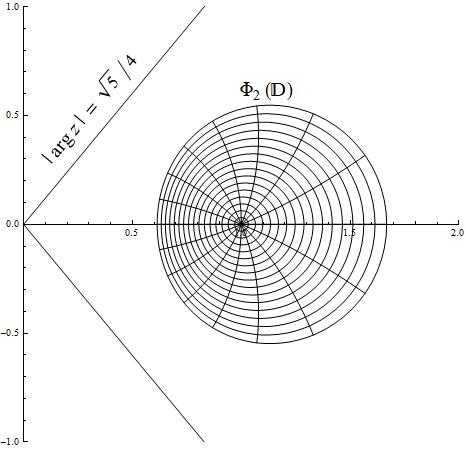}}
 \caption{Graph showing $\Phi_i(z)\prec ((1+z)/(1-z))^{\alpha_{i}}$ ($i=1,2)$ with $\alpha_{1}>1/\sqrt{6}$ and $\alpha_{2}>\sqrt{5}/4$.}\label{p6fig1}
 \end{figure}

 \end{example}

\begin{remark}\label{p6rmk4}
Taking $\alpha=1$ in Corollary ~\ref{p6cor6}, we get the following well known result \cite[Lemma 4.5a, p.\ 232]{ural}:

If $a, c \in \mathbb{R}$ such that $c>1+\sqrt{1+a^{2}}$, then $\RE \Phi(a;c;z)>0$.
\end{remark}

Next, we study the generalized and normalized Bessel function of the first kind of order $p$, $u_{p}(z)=u_{p,b,c}(z)$ given by the power series
 \[ u_{p}(z)= \sum_{n=0}^\infty\frac{(-c/4)^{n}}{(k)_{n}} \frac{z^n}{n!}, \]
where $b,p,c \in \mathbb{C}$ such that $k=p+(b+1)/2$ and $k \neq 0, -1, -2, \cdots$. The function $u_{p} \in \mathcal{H}[1,1]$ is a solution of the differential equation
\begin{equation}\label{p6eqn2.23}
4z^{2}u_{p}''(z) + 4kzu_{p}'(z)+czu_{p}(z)=0.
\end{equation}
The function $u_{p}(z)$ also satisfy the following recursive relation
\[ 4ku_{p}(z) = -cu_{p+1} (z),  \]
which is useful for studying its various geometric properties. More results regarding this function can be found in \cite{baricz2,baricz,baricz1}. By the use of Theorem ~\ref{p6thm2.1}, we obtain the following sufficient conditions for $u_{p}(z) \prec ((1+z)/(1-z))^{\alpha}$  $(0 < \alpha \leq 1)$.

\begin{corollary}\label{p6cor9}
Suppose that $0 < \alpha \leq 1$ and if $b,p,c \in \mathbb{R}$ satisfy the following condition
\begin{equation}\label{p6eqn2.41}
|c|<4\alpha(k-1)
\end{equation}
and $u_{p}(z)\neq 0$, then $u_{p}(z) \prec ((1+z)/(1-z))^{\alpha}$.
\end{corollary}

\begin{proof}
In view of ~\eqref{p6eqn2.23}, the function $u_{p}(z)$ satisfies ~\eqref{p6eqn2.1} with $C(z)=4$, $D(z)=4k$ and $E(z)=cz$. Also, note that $c \neq 0$ and $k>1$. For $z=x+iy \in \mathbb{D}$, we have
\begin{equation*}
\begin{split}
(\IM E(z))^{2}&- (\RE D(z)-C)^{2}\alpha^{2}\\
&\quad{}=c^{2}y^{2}-(4k-4)^{2}\alpha^{2}\\
&\quad{}<c^{2}(1-x^{2})-16(k-1)^{2}\alpha^{2}\\
&\quad{}=-c^{2}x^{2}-16(k-1)^{2}\alpha^{2}+c^{2}=:Px^{2}+Qx+R,
\end{split}
\end{equation*}
where $P=-c^{2}$, $Q=0$ and $R=-16(k-1)^{2}\alpha^{2}+c^{2}$. Clearly, $P<0$ and from ~\eqref{p6eqn2.41}, we get $Q^{2}-4PR<0$. So, $(\IM E(z))^{2}- (\RE D(z)-C)^{2}\alpha^{2}<0$ . Therefore, by applying Theorem ~\ref{p6thm2.1}, we conclude that $u_{p}(z) \prec ((1+z)/(1-z))^{\alpha}$.
\end{proof}

It is appealing to note that the above corollary is significant and leads to various  interesting relations involving the generalized and normalized Bessel function and  trigonometric functions by selecting the suitable choices of the parameters involved. The following example illustrates the Corollary ~\ref{p6cor9}.

\begin{example}
	Clearly, the following functions
	\[u_{2}(z)=u_{2,2,6}(z)=-\frac{5}{4 \sqrt{6}}\left(\frac{ (2 z-1) \sin (\sqrt{6 z})}{z^{5/2}}+\frac{\sqrt{6} \cos (\sqrt{6 z})}{ z^2}\right)\]
\begin{equation*}
\begin{split}
	u_{7}(z)=u_{7,6,10}(z)=&\frac{26.189163}{16}\Big(\frac{9\sqrt{10}\sin(\sqrt{10z})}{z^{19/2}}\Big(2000 z^4-61600 z^3+420420 z^2-720720 z\\
&+153153\Big)-\frac{10\cos(\sqrt{10z})}{z^{9}}  \Big(400 z^4-39600 z^3+540540 z^2-1891890 z\\
&+1378377\Big)\Big)
	\end{split}
\end{equation*}
	satisfy the hypothesis of Corollary ~\ref{p6cor9}. Therefore $u_i(z)\prec ((1+z)/(1-z))^{\alpha_{i}}$ (for $i=2,7)$ with $\alpha_{2}>3/5$ and $\alpha_{7}>5/19$. These subordinations are illustrated graphically in Figure \ref{p7fig1}.

\begin{figure}[htb!]
 \centering
 {\includegraphics[width=.3\linewidth]{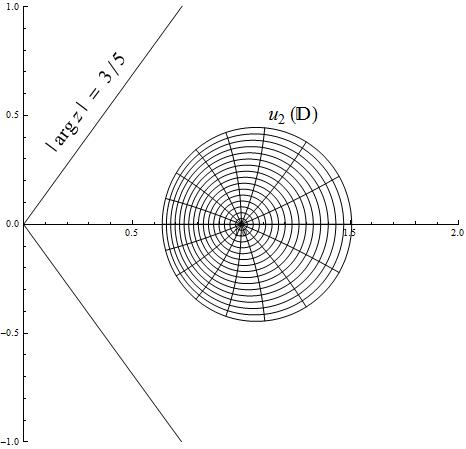}}
 \hspace{2cm}
 {\includegraphics[width=.3\linewidth]{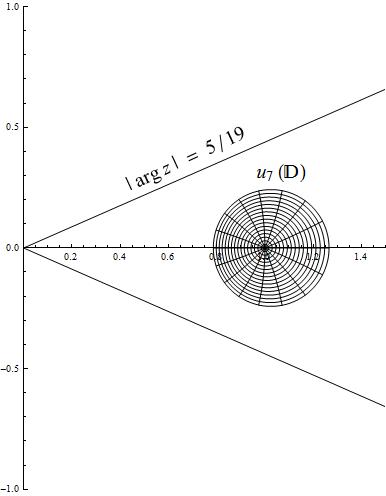}}
 \caption{Graph showing $u_i(z)\prec ((1+z)/(1-z))^{\alpha_{i}}$ ($i=2,7)$ with $\alpha_{2}>3/5$ and $\alpha_{7}>5/19$.}
 \label{p7fig1}
 \end{figure}

 \end{example}

\begin{remark}\label{p6rmk6}
For $\alpha=1$, Corollary ~\ref{p6cor9} reduces to the following result \cite[Theorem 2.2, p.\ 29]{baricz1}:

If $b,p,c \in \mathbb{R}$ such that $k>(|c|/4)+1$, then $\RE u_{p}(z)>0$ for all $z \in \mathbb{D}$.
\end{remark}

The following result gives the sufficient conditions for functions $h\in\mathcal{A}_{n}$ to belong to the class of strongly starlike functions of order $\alpha$.

\begin{corollary}\label{p6cor3}
Let $n$ be a positive integer, $0< \alpha \leq 1$ and $C(z)=C \geq 0$. Suppose the functions $D, E:\mathbb{D}\to\mathbb{C}$ satisfy
\[|\IM E(z)|< n \alpha(\RE D(z)-C).\]
If $h \in \mathcal{A}_{n}$ satisfy
\begin{equation*}
\begin{split}
C \Big(2\Big(\frac{zh'(z)}{h(z)}\Big)^{3}&-\frac{3z^{3}h'(z)h''(z)}{h^{2}(z)}+ \frac{z^{3} h'''(z)}{h(z)}\Big)\\
&+(2C+D(z)) \Big(\frac{z^{2}h''(z)}{h(z)}-\Big(\frac{zh'(z)}{h(z)}\Big)^{2}\Big)+(D(z)+ E(z))\frac{zh'(z)}{h(z)}=0,
\end{split}
\end{equation*}
then $h \in \mathcal{S}^*[\alpha]$.
\end{corollary}

\begin{proof}
Let the function $p:\mathbb{D}\to\mathbb{C}$ be defined by $p(z)=zh'(z)/h(z)$. Then $p$ is analytic in $\mathbb{D}$ with $p(0)=1$. A calculation shows that
\[\frac{zp'(z)}{p(z)}=1+\frac{zh''(z)}{h'(z)}-\frac{zh'(z)}{h(z)}. \]
The result now follows from Theorem ~\ref{p6thm2.1}.
\end{proof}

We obtain our next result by taking $n=1$, $C(z)=0,$ $D(z)=1$, $h(z)=f(z)g(z)$ with $f \in \mathcal{A}$, $g \in \mathcal{H}[1,1]$ and $E(z)=-1-zh''(z)/h'(z)+zh'(z)/h(z)$ in Corollary ~\ref{p6cor3}.

\begin{corollary}\label{p6cor4}
Let $0 < \alpha \leq 1$. Suppose that the functions $f \in \mathcal{A}$, $g \in \mathcal{H}[1,1]$ satisfy
 \[\Big|\IM\Big(1+\frac{zf''(z)g(z)+2zf'(z)g'(z)
 +zg''(z)f(z)}{f'(z)g(z)+g'(z)f(z)}\Big)
 -\left(\frac{zf'(z)}{f(z)} +\frac{zg'(z)}{g(z)}\right)\Big|<\alpha,\]
then $fg \in \mathcal{S}^*[\alpha]$.
\end{corollary}

\begin{remark}\label{p6rmk8}
For $g(z)\equiv1$, Corollary ~\ref{p6cor4} reduces to the following result:

Let $0 < \alpha \leq 1$. Suppose that the function $f \in \mathcal{A}$ satisfy
 \[\Big|\IM\Big(1+\frac{zf''(z)}{f'(z)}-\frac{zf'(z)}{f(z)} \Big)\Big|<\alpha,\]
then $f \in \mathcal{S}^*[\alpha]$.
\end{remark}

\section{Differential Subordination for  strong  starlikeness   of order $\alpha$}\label{p6sec3}

Let $p$ be an analytic function in $\mathbb{D}$ with $p(0)=1$. In the first two results, condition on $\beta$ is obtained so that the subordination
\[ 1+\beta \frac{zp'(z)}{p^{k}(z)} \prec  \varphi_{CAR}(z)\quad(k=1,2) \]
implies $p(z)\prec ((1+z)/(1-z))^{\alpha}$  $(0<\alpha\leq 1)$.

\begin{theorem}\label{p6thm3.1}Let $p$ be an analytic function defined on $\mathbb{D}$ with $p(0)=1$ satisfying
\[ 1+\beta \frac{zp'(z)}{p(z)} \prec \varphi_{CAR}(z). \]
If $|\beta|\geq (\sqrt{(4\sqrt{3}+8)/(3\sqrt{3})})/\alpha\simeq 1.6947/\alpha \ (0<\alpha\leq 1)$, then $p(z)\prec ((1+z)/(1-z))^{\alpha}$.
\end{theorem}
\begin{proof}
Define the function $q:\mathbb{D}\to\mathbb{C}$ by $q(z)=((1+z)/(1-z))^{\alpha} \ (0<\alpha\leq 1)$ with $q(0)=1$. Let us define $\varphi(w)=\beta/w$, $\nu(w)=1$ and
\[Q(z):=zq'(z)\varphi(q(z))=\frac{\beta zq'(z)}{q(z)}=\frac{2\alpha \beta z}{1-z^{2}}. \]
Since the mapping $z/(1-z^{2})$ maps $\mathbb{D}$ onto the entire plane minus the two half lines $1/2\leq y < \infty$ and $-\infty< y \leq -1/2$, $Q(z)$ is starlike univalent in $\mathbb{D}$. It follows from Lemma~\ref{p6int2}, that the subordination
\[ 1+\beta \frac{zp'(z)}{p(z)}\prec 1+\beta \frac{zq'(z)}{q(z)}  \]
implies  $p(z)\prec q(z)$.  The theorem is proved by computing $\beta$ so that
\begin{equation}\label{p6eqn3.1}
 \varphi_{CAR}(z):=1+\frac{4z}{3}+\frac{2z^{2}}{3}\prec 1+\beta \frac{zq'(z)}{q(z)}= 1+\frac{2\alpha \beta z}{1-z^{2}}:=h(z).
\end{equation}
Clearly, $\varphi_{CAR}(\mathbb{D})=\left\{w\in\mathbb{C}:|-2+\sqrt{6w-2}|<2 \right\}$. The subordination $\varphi_{CAR}(z)\prec h(z)$ holds if $\partial h(\mathbb{D})\subset \mathbb{C}\setminus \overline{\varphi_{CAR}(\mathbb{D})}$. Thus, by using the definition of $h$ as given in ~\eqref{p6eqn3.1}, the subordination $\varphi_{CAR}(z)\prec h(z)$ holds if for $t\in[-\pi,\pi]$, we have
\begin{equation}\label{p6eqn3.36}
\left|\sqrt{4+\frac{12\alpha\beta e^{i t}}{1-e^{2i t}}}-2\right|\geq 2.
\end{equation}
Set
\begin{equation}\label{p6eqn3.2}
w=u+iv=4+(12\alpha\beta e^{i t})/(1-e^{2i t}).
\end{equation}
  Then, condition ~\eqref{p6eqn3.36} holds if $|\sqrt{w}-2|\geq 2$ which is same as $|w|\geq 4\RE(\sqrt{w})$ or equivalently if
\begin{equation}\label{p6eqn3.3}
 (u^{2}+v^{2}-8u)^{2}-64(u^{2}+v^{2})\geq 0.
\end{equation}
From ~\eqref{p6eqn3.2}, we get $u=4$ and $v=(6\alpha\beta)/\sin t$. Using these, we see that ~\eqref{p6eqn3.3} reduces to
\[g(t):=\frac{48}{\sin^{4} t}  (-16 \sin^{4} t-72\alpha^{2}\beta^{2}\sin^{2} t+27 \alpha^{4}\beta^{4})\geq 0 \quad (0\leq t \leq 2\pi).\]
Since $g(t)=g(-t)$, we restrict to $0\leq t \leq \pi$. Writing $x= \sin t$, the problem reduces to finding the values of $\beta$ for which
\[ f(x):=-16x^{4}-72\alpha^{2}\beta^{2}x^{2} +27 \alpha^{4}\beta^{4}\geq 0 \quad (0\leq x\leq 1).\]
The function $f$ is decreasing as $f'(x)=-144\alpha^{2}\beta^{2}x-64x^{3} <0$ and therefore, $f(x)\geq f(1)=-16-72\alpha^{2}\beta^{2} +27 \alpha^{4}\beta^{4} \geq 0$ using the hypothesis that \[|\beta|\geq \big(\sqrt{(4\sqrt{3}+8)/(3\sqrt{3})}\big)/\alpha\simeq1.6947/\alpha.\qedhere\]
\end{proof}

\begin{theorem}\label{p6thm3.2}
Let $p$ be an analytic function defined on $\mathbb{D}$ with $p(0)=1$ satisfying
\[ 1+\beta \frac{zp'(z)}{p^{2}(z)} \prec \varphi_{CAR}(z), \]
then the following results hold:
\begin{enumerate}[(a)]
  \item If $\beta\geq 4$ or $\beta\leq -4/3$, then $p(z)\prec (1+z)/(1-z)$.
  \item If $\beta>0$ and $0<\alpha<1$ satisfy
  \begin{equation}\label{p6eqn3.8}
\begin{split}
  9\alpha^{2}\beta^{2}\Big(\frac{1-\alpha}{1+\alpha} \Big)^{\alpha}(1-\alpha^{2})^{-2}&\Big(-8+\alpha^{2}\Big(8+3\beta^{2}\Big(\frac{1-\alpha}{1+\alpha} \Big)^{\alpha} \Big) \Big)\\
  &-64\alpha\beta(1-\alpha^{2})^{-\frac{1}{2}}\Big(\sqrt{\frac{1-\alpha}{1+\alpha}} \Big)^{\alpha}\sin \big(\frac{\alpha\pi}{2} \big)\geq 16,
 \end{split}
   \end{equation}
   then $p(z)\prec ((1+z)/(1-z))^{\alpha}$.
\end{enumerate}
\end{theorem}

\begin{proof}
Define the function $q:\mathbb{D}\to\mathbb{C}$ by $q(z)=((1+z)/(1-z))^{\alpha} \ (0<\alpha\leq1)$. Let us define $\varphi(w)=\beta/w^{2}$, $\nu(w)=1$ and $Q(z)=zq'(z)\varphi(q(z))=(2\alpha\beta z(1-z)^{\alpha})/((1-z^{2})(1+z)^{\alpha})$. A computation shows that
\[\frac{zQ'(z)}{Q(z)}=\frac{1+z^{2}-2\alpha z}{1-z^{2}}. \]
Let $z=r e^{i t}, -\pi\leq t \leq \pi, 0\leq r<1$. Then
\[ \RE\left(\frac{1+z^{2}-2\alpha z}{1-z^{2}}\right)=\frac{(1-r^{2})(1+r^{2}-2\alpha r\cos t)}{(1-r^{2}\cos 2t)^{2}+(r^{2}\sin 2t)^{2}}. \]
Since $0<\alpha \leq 1,$ $0\leq r<1$, we have $1+r^{2}-2\alpha r\cos t \geq 1+r^{2}-2\alpha r\geq (1-\alpha r)^{2}> 0$ and so, it follows that $Q$ is starlike in $\mathbb{D}$. As a consequence, the subordination
\[ 1+\beta \frac{zp'(z)}{p^{2}(z)}\prec 1+\beta \frac{zq'(z)}{q^{2}(z)}  \]
implies $p(z)\prec q(z)$ by Lemma~\ref{p6int2}. We complete the proof by showing
\begin{equation}\label{p6eqn3.5}
\varphi_{CAR}(z):=1+\frac{4z}{3}+\frac{2z^{2}}{3}\prec 1+\beta \frac{zq'(z)}{q^{2}(z)}= 1+\frac{2\alpha\beta z(1-z)^{\alpha}}{(1-z^{2})(1+z)^{\alpha}}:=h(z).
\end{equation}
 Then $\varphi_{CAR}(\mathbb{D})=\left\{w\in\mathbb{C}:|-2+\sqrt{6w-2}|<2 \right\}$. The subordination $\varphi_{CAR}(z)\prec h(z)$ holds if $\partial h(\mathbb{D})\subset \mathbb{C}\setminus \overline{\varphi_{CAR}(\mathbb{D})}$. Thus, by using the definition of $h$ as given in ~\eqref{p6eqn3.5}, the subordination $\varphi_{CAR}(z)\prec h(z)$ holds if for $t\in[-\pi,\pi]$, we have
\[ \left|\sqrt{4+\frac{12\alpha \beta e^{i t}(1-e^{i t})^{\alpha}}{(1-e^{2i t})(1+e^{i t})^{\alpha}}}-2\right|\geq 2. \]
Let
\begin{equation}\label{p6eqn3.6}
 w=u+iv=4+\frac{12\alpha \beta e^{i t}(1-e^{i t})^{\alpha}}{(1-e^{2i t})(1+e^{i t})^{\alpha}}.
 \end{equation}
 Then, proceeding as in Theorem~\ref{p6thm3.1}, we have to show ~\eqref{p6eqn3.3}. From \eqref{p6eqn3.6}, we get
 \[u=4+\frac{6 \alpha \beta \sin(\alpha \pi/2)(\sin t/2)^{\alpha}}{\sin t (\cos t/2)^{\alpha}},\quad v=\frac{6 \alpha \beta \cos(\alpha\pi/2)(\sin t/2)^{\alpha}}{\sin t (\cos t/2)^{\alpha}}\]
and using these, we see that the inequality ~\eqref{p6eqn3.3} reduces to
 \begin{equation*}
\begin{split}
  -16 + \alpha\beta \csc t (\tan^{\alpha} t/2)(-64\sin(\alpha \pi/2)&+9\alpha\beta\csc t(\tan^{\alpha} t/2)(-8 \\
  &+3  \alpha^{2}\beta^{2}\csc^{2}t \tan^{2 \alpha} t/2   ))\geq 0.
\end{split}
\end{equation*}
Under the given hypothesis, we show that $f(x)\geq 0$ for $0\leq x\leq 1$, where $x=\cos t/2$ and
\begin{equation}\label{p6eqn3.7}
\begin{split}
  f(x)=-16 +&\frac{\alpha\beta}{2x \sqrt{1-x^{2}}} \Big(\frac{\sqrt{1-x^{2}}}{x}\Big)^{\alpha}\Big(-64\sin(\alpha \pi/2)\\  &+\frac{9\alpha\beta}{2x\sqrt{1-x^{2}}}\Big(\frac{\sqrt{1-x^{2}}}{x}\Big)^{\alpha}
  \Big(-8+\frac{3  \alpha^{2}\beta^{2}}{4 x^{2}(1-x^{2})}\Big(\frac{1-x^{2}}{x^{2}}\Big)^{\alpha}  \Big)\Big).
\end{split}
\end{equation}

     (a) If $\alpha=1$, then ~\eqref{p6eqn3.7} reduces to $f(x)=(4x^{2}+3\beta)^{3}(\beta-4x^{2})/16x^{8}$. Clearly, $f(x)\geq 0$ if $\beta \leq -4/3$ or $\beta \geq 4$.

     (b) Assume that $0<\alpha<1$. A simple computation shows that $f'(x)=0$ if $x=\pm\sqrt{(1+\alpha)/2}$. Then $\sqrt{(1+\alpha)/2}\in[0,1]$  and $f''(\sqrt{(1+\alpha)/2})>0$ if $\beta>0$. Therefore, $f(x)\geq f(\sqrt{(1+\alpha)/2})$ for $\beta>0$. Observe that
     \begin{equation*}
\begin{split}
  f(\sqrt{(1+\alpha)/2}) =-16 &+ 9\alpha^{2}\beta^{2}\Big(\frac{1-\alpha}{1+\alpha} \Big)^{\alpha}(1-\alpha^{2})^{-2}\Big(-8+\alpha^{2}\Big(8+3\beta^{2}\Big(\frac{1-\alpha}{1+\alpha} \Big)^{\alpha} \Big) \Big)\\
&-64\alpha\beta(1-\alpha^{2})^{-1/2}\Big(\sqrt{\frac{1-\alpha}{1+\alpha}} \Big)^{\alpha}\sin \big(\frac{\alpha\pi}{2} \big)
\end{split}
\end{equation*}
and $f(\sqrt{(1+\alpha)/2})\geq 0$ by ~\eqref{p6eqn3.8}. This completes our proof.
\end{proof}

The next result depicts the condition on $\beta$ so that the subordination
\[ p(z)+\beta \frac{zp'(z)}{p^{k}(z)} \prec\Big(\frac{1+z}{1-z}\Big)^{\alpha}\quad (k=0,1,2)  \]
implies $p(z)\prec ((1+z)/(1-z))^{\alpha}$, where $p$ is an analytic function in $\mathbb{D}$ with $p(0)=1$ and $0<\alpha\leq 1$.

\begin{theorem}\label{p6thm3.4}
Let $p$ be an analytic function defined on $\mathbb{D}$ with $p(0)=1$. Then following are the sufficient conditions for $p(z)\prec ((1+z)/(1-z))^{\alpha} \ (0<\alpha\leq 1)$.
\begin{enumerate}[(a)]
  \item The function $p$ satisfies the subordination
  \[ p(z)+\beta zp'(z) \prec \Big(\frac{1+z}{1-z}\Big)^{\alpha} \quad (\beta>0). \]
  \item The function $p$ satisfies the subordination
  \[ p(z) +\beta \frac{zp'(z)}{p(z)} \prec \Big(\frac{1+z}{1-z}\Big)^{\alpha} \quad (\beta>0).  \]
  \item The function $p$ satisfies the subordination
  \[ p(z) +\beta \frac{zp'(z)}{p^{2}(z)} \prec \Big(\frac{1+z}{1-z}\Big)^{\alpha}\quad (\text{either}\quad 0<\alpha<\frac{1}{2}, \beta>0 \quad \text{or}\quad \frac{1}{2}<\alpha\leq 1, \beta<0).  \]
\end{enumerate}
\end{theorem}

\begin{proof}
Define the function $q:\mathbb{D}\to\mathbb{C}$ by $q(z)=((1+z)/(1-z))^{\alpha} \ ( 0<\alpha\leq 1)$, with $q(0)=1$. Clearly, the function $q$ is convex in $\mathbb{D}$.

   (a) Let us define $\phi(w)=\beta$ $(\beta>0)$. Since $\RE\phi(q(z))=\beta>0$, the result now follows from Lemma~\ref{p6int3}.

   (b) Define the function $\phi$ as $\phi(w)=\beta/w$ $(\beta>0)$. Consider
      \[\RE\phi(q(z))=\beta \RE\left(\frac{1}{q(z)}\right)>0.\]
      The result now follows from Lemma~\ref{p6int3}.

     (c) We define the function $\phi$ as $\phi(w)=\beta/w^{2}$. A simple computation shows that
      \[\phi(q(z))=\beta \left(\frac{1}{q^{2}(z)}\right)=\beta\Big(\frac{1-z}{1+z}\Big)^{2 \alpha}=:M(z).\]
      and for $-\pi\leq t \leq \pi$,
      \[\RE(M(e^{i t}))=\beta \cos(\alpha \pi)\Big(\frac{\sin t/2}{\cos t/2}\Big)^{2 \alpha}.  \]
      It is easy to see that $\RE(M(e^{i t}))>0$ if either $0<\alpha<1/2, \beta>0$ or $1/2<\alpha\leq 1, \beta<0$. This completes the proof by the use of Lemma~\ref{p6int3}.
\end{proof}

We will illustrate the above three theorems by the following examples. By taking $p(z)=zf'(z)/f(z)$ in Theorems ~\ref{p6thm3.1}, ~\ref{p6thm3.2} and ~\ref{p6thm3.4} gives the following example:

\begin{example}\label{p6eg3.1}
Let $f\in \mathcal{A}$ and $0<\alpha\leq 1$. The following are some sufficient conditions for the function $f\in \mathcal{S}^*[\alpha]:$
\begin{enumerate}
  \item  The function $f$ satisfies the subordination
  \[1+\beta\left(1+\frac{zf''(z)}{f'(z)}-\frac{zf'(z)}{f(z)}\right)\prec \varphi_{CAR}(z)\quad \Big(|\beta|\geq \frac{1}{\alpha}\sqrt{\frac{4\sqrt{3}+8}{3\sqrt{3}}}\simeq \frac{1.6947}{\alpha} \Big). \]

  \item The function $f$ satisfies the subordination
  \[ 1-\beta+\beta\frac{1+\frac{zf''(z)}{f'(z)}}{\frac{zf'(z)}{f(z)}}\prec \varphi_{CAR}(z).\]
   where $0<\alpha<1$ and $\beta>0$ satisfy
   \begin{equation*}
  \begin{split}
  9\alpha^{2}\beta^{2}\Big(\frac{1-\alpha}{1+\alpha} \Big)^{\alpha}(1-\alpha^{2})^{-2}&\Big(-8+\alpha^{2}\Big(8+3\beta^{2}\Big(\frac{1-\alpha}{1+\alpha} \Big)^{\alpha} \Big) \Big)\\
  &-64\alpha\beta(1-\alpha^{2})^{-\frac{1}{2}}\Big(\sqrt{\frac{1-\alpha}{1+\alpha}} \Big)^{\alpha}\sin \big(\frac{\alpha\pi}{2} \big)\geq 16.
 \end{split}
 \end{equation*}

\item  The function $f$ satisfies the subordination
  \[\frac{zf'(z)}{f(z)}+\beta\frac{zf'(z)}{f(z)}\left(1+\frac{zf''(z)}{f'(z)}-\frac{zf'(z)}{f(z)}\right)\prec \Big(\frac{1+z}{1-z}\Big)^{\alpha}\quad (\beta>0). \]

\item  The function $f$ satisfies the subordination
  \[(1-\beta)\frac{zf'(z)}{f(z)}+\beta\left(1+\frac{zf''(z)}{f'(z)}\right)\prec \Big(\frac{1+z}{1-z}\Big)^{\alpha}\quad (\beta>0). \]

  \item The function $f$ satisfies the subordination
  \[ \frac{zf'(z)}{f(z)}-\beta+\beta\frac{1+\frac{zf''(z)}{f'(z)}}{\frac{zf'(z)}{f(z)}}\prec \Big(\frac{1+z}{1-z}\Big)^{\alpha}\] whenever either $ 0<\alpha<\frac{1}{2}, \beta>0$ or $ \frac{1}{2}<\alpha\leq 1, \beta<0$.

\end{enumerate}
\end{example}

By taking $p(z)=f'(z)$ in Theorems ~\ref{p6thm3.1}, ~\ref{p6thm3.2} and ~\ref{p6thm3.4} gives the following example:

\begin{example}\label{p6eg3.2}
Let $f\in \mathcal{A}$ and $0<\alpha\leq 1$. The following are the sufficient conditions for the function $f'(z)\prec ((1+z)/(1-z))^{\alpha}:$
\begin{enumerate}
  \item  The function $f$ satisfies the subordination
  \[1+\beta\left(\frac{zf''(z)}{f'(z)}\right)\prec \varphi_{CAR}(z)\quad \Big(|\beta|\geq \frac{1}{\alpha}\sqrt{\frac{4\sqrt{3}+8}{3\sqrt{3}}}\simeq \frac{1.6947}{\alpha} \Big). \]

  \item The function $f$ satisfies the subordination
  \[ 1+\beta\left(\frac{zf''(z)}{(f'(z))^{2}}\right)\prec \varphi_{CAR}(z),\]
   where $0<\alpha<1$ and $\beta>0$ satisfy
   \begin{equation*}
\begin{split}
  9\alpha^{2}\beta^{2}\Big(\frac{1-\alpha}{1+\alpha} \Big)^{\alpha}(1-\alpha^{2})^{-2}&\Big(-8+\alpha^{2}\Big(8+3\beta^{2}\Big(\frac{1-\alpha}{1+\alpha} \Big)^{\alpha} \Big) \Big)\\
  &-64\alpha\beta(1-\alpha^{2})^{-\frac{1}{2}}\Big(\sqrt{\frac{1-\alpha}{1+\alpha}} \Big)^{\alpha}\sin \big(\frac{\alpha\pi}{2} \big)\geq 16.
 \end{split}
   \end{equation*}

   \item  The function $f$ satisfies the subordination
  \[f'(z)+\beta zf''(z)\prec \Big(\frac{1+z}{1-z}\Big)^{\alpha}\quad (\beta>0). \]

   \item  The function $f$ satisfies the subordination
  \[f'(z)+\beta\left(\frac{zf''(z)}{f'(z)}\right)\prec \Big(\frac{1+z}{1-z}\Big)^{\alpha}\quad (\beta>0). \]

  \item The function $f$ satisfies the subordination
  \[ f'(z)+\beta\left(\frac{zf''(z)}{(f'(z))^{2}}\right)\prec \Big(\frac{1+z}{1-z}\Big)^{\alpha}\quad (\text{either}\quad 0<\alpha<\frac{1}{2}, \beta>0 \quad \text{or}\quad \frac{1}{2}<\alpha\leq 1, \beta<0).  \]

\end{enumerate}
\end{example}

The five parts of the next example are obtained by taking $p(z)=z^{2}f'(z)/f^{2}(z)$ in Theorems ~\ref{p6thm3.1}, ~\ref{p6thm3.2} and ~\ref{p6thm3.4} respectively.

\begin{example}\label{p6eg3.3}
Let $f\in \mathcal{A}$ and $0<\alpha\leq1$. The following are the sufficient conditions for the function $z^{2}f'(z)/f^{2}(z)\prec((1+z)/(1-z))^{\alpha}:$
\begin{enumerate}
\item The function $f$ satisfies
  \[ 1+\beta\left(\frac{(zf(z))''}{f'(z)}-\frac{2zf'(z)}{f(z)}\right)\prec \varphi_{CAR}(z)\quad \Big(|\beta|\geq \frac{1}{\alpha}\sqrt{\frac{4\sqrt{3}+8}{3\sqrt{3}}}\simeq \frac{1.6947}{\alpha} \Big). \]
\item The function $f$ satisfies the subordination
  \[ 1+\beta \frac{f^{2}(z)}{z^{2}f'(z)} \left(\frac{(zf(z))''}{f'(z)}-\frac{2zf'(z)}{f(z)}\right)\prec \varphi_{CAR}(z).\]
   where $0<\alpha<1$ and $\beta>0$ satisfy
   \begin{equation*}
\begin{split}
  9\alpha^{2}\beta^{2}\Big(\frac{1-\alpha}{1+\alpha} \Big)^{\alpha}(1-\alpha^{2})^{-2}&\Big(-8+\alpha^{2}\Big(8+3\beta^{2}\Big(\frac{1-\alpha}{1+\alpha} \Big)^{\alpha} \Big) \Big)\\
  &-64\alpha\beta(1-\alpha^{2})^{-\frac{1}{2}}\Big(\sqrt{\frac{1-\alpha}{1+\alpha}} \Big)^{\alpha}\sin \big(\frac{\alpha\pi}{2} \big)\geq 16.
 \end{split}
   \end{equation*}

   \item The function $f$ satisfies
  \[\frac{z^{2}f'(z)}{f^{2}(z)} +\beta\frac{z^{2}f'(z)}{f^{2}(z)}\left(\frac{(zf(z))''}{f'(z)}-\frac{2zf'(z)}{f(z)}\right)\prec \Big(\frac{1+z}{1-z}\Big)^{\alpha}\quad (\beta>0). \]

   \item The function $f$ satisfies
  \[\frac{z^{2}f'(z)}{f^{2}(z)} +\beta\left(\frac{(zf(z))''}{f'(z)}-\frac{2zf'(z)}{f(z)}\right)\prec \Big(\frac{1+z}{1-z}\Big)^{\alpha}\quad (\beta>0). \]

\item Either $0<\alpha<1/2, \beta>0$ or $1/2<\alpha\leq 1, \beta<0$ and the function $f$ satisfies the subordination
  \[ \frac{z^{2}f'(z)}{f^{2}(z)}+\beta \frac{f^{2}(z)}{z^{2}f'(z)} \left(\frac{(zf(z))''}{f'(z)}-\frac{2zf'(z)}{f(z)}\right)\prec \Big(\frac{1+z}{1-z}\Big)^{\alpha}.  \]

\end{enumerate}
\end{example}

Let $p$ be an analytic function in $\mathbb{D}$ with $p(0)=1$. In the next result, condition on $\beta$ is obtained so that the subordination $p(z)+\beta zp'(z) \prec \varphi_{CAR}(z)$ implies $p(z)\prec (1+z)/(1-z)$.

\begin{theorem}\label{p6thm3.5}Let $p$ be an analytic function defined on $\mathbb{D}$ with $p(0)=1$ satisfying
\[ p(z)+\beta zp'(z) \prec \varphi_{CAR}(z) \quad\text{for}\quad \beta \geq 0 .\]
Then $p(z)\prec (1+z)/(1-z)$.
\end{theorem}

\begin{proof}
Obviously, the result hold for $\beta=0$. Let us assume that $\beta >0$. Define the function $q:\mathbb{D}\to\mathbb{C}$ by $q(z)=(1+z)/(1-z)$. Also define $\nu(w):=w$ and $\varphi(w):=\beta$. Note that the functions $\nu$ and $\varphi$ are analytic in $\mathbb{C}$ and $\varphi(w)\neq 0$. Consider the functions $Q$ and $h$ defined as follows:
\[Q(z):=zq'(z)\varphi(q(z))=\beta zq'(z)=\frac{2\beta z}{(1-z)^{2}} \]
and
\[h(z):=\nu(q(z))+Q(z)=q(z)+Q(z). \]
The function $q$ is convex in $\mathbb{D}$ and therefore the function $Q$ is starlike univalent in $\mathbb{D}$ which further implies that
\[\RE\Big(\frac{zh'(z)}{Q(z)}\Big)=\frac{1}{\beta}+\RE\Big(\frac{zQ'(z)}{Q(z)}\Big)>0.\]
Since all the conditions of Lemma~\ref{p6int2} are fulfilled, it follows that $p(z)\prec q(z)$. We complete the proof by showing that
\begin{equation}\label{p6eqn3.9}
\varphi_{CAR}(z):=1+\frac{4z}{3}+\frac{2z^{2}}{3}\prec q(z)+\beta zq'(z)= \frac{1+z}{1-z}+\frac{2\beta z}{(1-z)^{2}}:=h(z).
\end{equation}
Proceeding as in Theorem ~\ref{p6thm3.1} and by using the definition of $h$ as given in ~\eqref{p6eqn3.9}, the subordination $\varphi_{CAR}(z)\prec h(z)$ holds if the following condition holds:
\[\left|\sqrt{-2+\frac{12\beta e^{i t}}{(1-e^{i t})^{2}}+\frac{6(1+e^{i t})}{1-e^{i t}}}-2\right|\geq 2 \quad (-\pi\leq t\leq \pi). \]
Set
\[w=u+iv=-2+\frac{12\beta e^{i t}}{(1-e^{i t})^{2}}+\frac{6(1+e^{i t})}{1-e^{i t}}\]
so that
\[ u=-2 - \frac{3 \beta}{\sin^{2}(t/2)} \quad\text{and}\quad v=\frac{6\cos (t/2)}{\sin (t/2)}.\]
 Then, substituting the values of $u$ and $v$ in ~\eqref{p6eqn3.3}, we need to prove that, for $-\pi\leq t\leq \pi$
\[g(t):=\frac{-768}{2\sin^{2}t/2}(6+5\beta+3\cos t) + \frac{432}{\sin^{4}t/2}(3+\beta(6+\beta))  +\frac{648\beta^{2}(1+\beta)}{\sin^{6}t/2}+\frac{81\beta^{4}}{\sin^{8}t/2} \geq 0.\]
Observe that $g(t)=g(-t)$ and after substituting $x=\sin(t/2)$, we see that the above inequality holds if $f(x)\geq 0$ for all $x$ with $0\leq x \leq 1$, where
\[f(x)=\frac{3}{x^8} \big(27\beta^{4}+216\beta^{2}(1+\beta)x^{2}+144(3+\beta(6+\beta))x^{4}+128(-9-5\beta+6x^{2})x^{6}\big).\]
 A simple computation shows that the function $f$ is decreasing and therefore, $f(x)\geq f(1)=3(6+\beta)(2+3\beta)^{3} >0$  for $\beta>0$. Hence, $p(z)\prec q(z)$.
\end{proof}

The three parts of the next example are obtained by taking $p(z)=zf'(z)/f(z)$, $p(z)=f'(z)$ and $p(z)=z^{2}f'(z)/f^{2}(z)$ in Theorem ~\ref{p6thm3.5} respectively.

\begin{example}\label{p6eg3.3a}
Let $f\in \mathcal{A}$ and $\beta\geq 0$.
\begin{enumerate}
\item If the function $f$ satisfies the subordination
  \[\frac{zf'(z)}{f(z)}+\frac{\beta zf'(z)}{f(z)}\left(1+\frac{zf''(z)}{f'(z)}-\frac{zf'(z)}{f(z)}\right)\prec \varphi_{CAR}(z), \]
  then $f\in \mathcal{S}^*$.
\item If the function $f$ satisfies the subordination
  \[ f'(z)+ \beta z f''(z)\prec \varphi_{CAR}(z),\]
  then $f'(z)\in \mathcal{P}$.
\item If the function $f$ satisfies the subordination
  \[ \frac{z^{2}f'(z)}{f^{2}(z)}+\frac{\beta z^{2}f'(z)}{f^{2}(z)}\left(\frac{(zf(z))''}{f'(z)}-\frac{2zf'(z)}{f(z)}\right)\prec \varphi_{CAR}(z),\]
  then $z^{2}f'(z)/f^{2}(z)\in \mathcal{P}$.

\end{enumerate}
\end{example}

Let $p$ be an analytic function in $\mathbb{D}$ with $p(0)=1$. In the next result, condition on $\beta$ is obtained so that the subordination
\[ p(z)+\beta \frac{zp'(z)}{p^{2}(z)} \prec \varphi_{CAR}(z)  \]
implies $p(z)\prec (1+z)/(1-z)$.

\begin{theorem}\label{p6thm3.3}Let $p$ be an analytic function defined on $\mathbb{D}$ with $p(0)=1$ satisfying
\[ p(z)+\beta \frac{zp'(z)}{p^{2}(z)} \prec \varphi_{CAR}(z) \quad\text{for}\quad \beta \leq 0. \]
Then $p(z)\prec (1+z)/(1-z)$.
\end{theorem}

\begin{proof}
For $\beta=0$, result hold obviously. Let us assume that $\beta <0$. Define the function $q:\mathbb{D}\to\mathbb{C}$ by $q(z)=(1+z)/(1-z)$. Also define the functions $\nu$ and $\varphi$ by $\nu(w)=w$ and $\varphi(w)=\beta/w^{2}$. Clearly, the functions $\nu$ is analytic in $\mathbb{C}$, $\varphi$ is analytic in $\mathbb{C}\setminus\{0\}$ and $\varphi(w)\neq 0$. Consider the functions $Q$ and $h$ defined as follows:
\[Q(z):=zq'(z)\varphi(q(z))=\frac{\beta zq'(z)}{q^{2}(z)}=\frac{2\beta z}{(1+z)^{2}} \]
and
\[h(z):=\nu(q(z))+Q(z)=q(z)+Q(z). \]
Observe that
\[\RE\Big(\frac{z Q'(z)}{Q(z)}\Big)=\RE\Big(\frac{1-z}{1+z}\Big)>0.   \]
This shows that $Q(z)$ is starlike univalent in $\mathbb{D}$. A computation shows that
\[\frac{zh'(z)}{Q(z)}=\frac{q^{2}(z)}{\beta}+\frac{zQ'(z)}{Q(z)}=\frac{1}{\beta}\left(\frac{1+z}{1-z}\right)^{2} +\frac{1-z}{1+z}=: H(z)\]
and
\[\RE (H(e^{i t}))= -(\cot^{2}(t/2))/\beta >0. \]
This condition together with the minimum principle for harmonic functions shows that $\RE(zh'(z)/Q(z))>0$. Since all the conditions of Lemma~\ref{p6int2} are fulfilled, it follows that $p(z)\prec q(z)$. We complete the proof by showing that
\[\varphi_{CAR}(z):=1+\frac{4z}{3}+\frac{2z^{2}}{3}\prec q(z)+\beta \frac{zq'(z)}{q^{2}(z)}= \frac{1+z}{1-z}+\frac{2\beta z}{(1+z)^{2}}:=h(z). \]
The subordination $\varphi_{CAR}(z) \prec h(z)$ holds if the following condition holds:
\[\left|\sqrt{-2+\frac{12\beta e^{i t}}{(1+e^{i t})^{2}}+\frac{6(1+e^{i t})}{1-e^{i t}}}-2\right|\geq 2 \quad (-\pi\leq t\leq \pi). \]
Set
\[w=u+iv=-2+\frac{12\beta e^{i t}}{(1+e^{i t})^{2}}+\frac{6(1+e^{i t})}{1-e^{i t}}\]
so that
\[ u=-2 + \frac{3 \beta}{\cos^{2}(t/2)} \quad\text{and}\quad v=\frac{6\cos (t/2)}{\sin (t/2)}.\]
 Then, substituting these values of $u$ and $v$, we see that the inequality ~\eqref{p6eqn3.3} is equivalent to
\begin{equation*}
\begin{split}
3 \Big(\frac{27 \beta^3}{\cos ^8(t/2)}& (\beta-8 \cos^{2} (t/2))+\frac{8 \beta}{\cos^{2} (t/2)} \left(\frac{45 \beta }{\cos^{2} (t/2)}+27\beta -28\right)\\
+&\frac{72}{\sin ^2(t/2)} (3 (\beta-4) \beta-16) +\frac{432}{\sin ^4(t/2)} +768\Big) \geq 0\quad (-\pi\leq t\leq \pi).
\end{split}
\end{equation*}
Substituting $x=\cos(t/2)$, the above inequality becomes $f(x)\geq 0$ for all $x$ with $0\leq x \leq 1$, where
\begin{equation*}
\begin{split}
f(x)=&\frac{1}{x^8 \left(1-x^{2}\right)^2} \Big(27 \beta^4 \left(1-x^{2}\right)^2-216 \beta^3 \left(1-x^{2}\right)^2 x^2+72 \beta^2 \left(2 x^4-7 x^2+5\right) x^4\\
&-32 \beta (1-x^{2}) \left(20 x^2+7\right) x^6+48 \left(1-4 x^2\right)^2 x^8\Big).
\end{split}
\end{equation*}
Since $g(x):=2 x^4-7 x^2+5\geq g(1)=0,$ it follows that $f(x)\geq 0$ for $\beta<0$. Hence, $p(z)\prec q(z)$.
\end{proof}

The three parts of the next example are obtained by taking $p(z)=zf'(z)/f(z)$, $p(z)=f'(z)$ and $p(z)=z^{2}f'(z)/f^{2}(z)$ in Theorem ~\ref{p6thm3.3} respectively.

\begin{example}\label{p6eg3.3b}
Let $f\in \mathcal{A}$ and $\beta\leq 0$.
\begin{enumerate}
\item If the function $f$ satisfies the subordination
  \[\frac{zf'(z)}{f(z)}+\frac{\beta f(z)}{zf'(z)}\left(1+\frac{zf''(z)}{f'(z)}-\frac{zf'(z)}{f(z)}\right)\prec \varphi_{CAR}(z), \]
  then $f\in \mathcal{S}^*$.
\item If the function $f$ satisfies the subordination
  \[ f'(z)+\frac{\beta z f''(z)}{(f'(z))^{2}}\prec \varphi_{CAR}(z),\]
  then $f'(z)\in \mathcal{P}$.
\item If the function $f$ satisfies the subordination
  \[ \frac{z^{2}f'(z)}{f^{2}(z)}+\frac{\beta f^{2}(z)}{z^{2}f'(z)}\left(\frac{(zf(z))''}{f'(z)}-\frac{2zf'(z)}{f(z)}\right)\prec \varphi_{CAR}(z),\]
  then $z^{2}f'(z)/f^{2}(z)\in \mathcal{P}$.
\end{enumerate}
\end{example}

\section*{Acknowledgement}The second author was supported by the Basic Science Research Program through the National Research Foundation of Korea (NRF) funded by the Ministry of Education, Science, and Technology (No. 2019R1I1A3A01050861).


\begin{thebibliography}{99}

\bibitem{AE} {A. P. Acharya}, {Univalence criteria for analytic funtions and applications to hypergeometric functions},
Ph.D. Diss., University of W\"{u}rzburg, 1997.

\bibitem{ali1} R.\ M.\ Ali, N.\ E.\ Cho, N.\ K.\ Jain and V.\ Ravichandran, Radii of starlikeness and convexity of functions defined by subordination with fixed second coefficients,  Filomat    \textbf{26} (2012), 553–-561.

\bibitem{ali2}R. M. Ali, N.\ E.\ Cho, V.\ Ravichandran and S. Sivaprasad Kumar, Differential subordination for functions associated with the lemniscate of Bernoulli, Taiwanese J. Math. {\bf 16} (2012), no.~3, 1017--1026. 

\bibitem{saiful} R. M. Ali, S. R. Mondal\ and\ V. Ravichandran, On the Janowski convexity and starlikeness of the confluent hypergeometric function, Bull. Belg. Math. Soc. Simon Stevin {\bf 22} (2015), no.~2, 227--250. 

\bibitem{baricz1}\'A. Baricz, {\it Generalized Bessel Functions of the First Kind}, Lecture Notes in Mathematics, 1994, Springer, Berlin, 2010. 

\bibitem{baricz} \'A. Baricz, Applications of the admissible functions method for some differential equations, Pure Math. Appl. {\bf 13} (2002), no.~4, 433--440. 

\bibitem{baricz2}\'A. Baricz, Geometric properties of generalized Bessel functions, Publ. Math. Debrecen {\bf 73} (2008), no.~1--2, 155--178. 
\bibitem{bohra}N. Bohra\ and\ V. Ravichandran, On confluent hypergeometric functions and generalized Bessel functions, Anal. Math. {\bf 43} (2017), no.~4, 533--545. 

\bibitem{cho}N. E. Cho, O. S. Kwon\ and\ V. Ravichandran, Coefficient, distortion and growth inequalities for certain close-to-convex functions, J. Inequal. Appl. {\bf 2011}, 2011:100, 7 pp. 

\bibitem{goodman}A. W. Goodman, {\it Univalent Functions. Vol. II}, Mariner, Tampa, FL, 1983. 

\bibitem{jano}W. Janowski, Extremal problems for a family of functions with positive real part and for some related families, Ann. Polon. Math. {\bf 23} (1970/1971), 159--177.

\bibitem{lee}S. K. Lee, V. Ravichandran\ and\ S. Supramaniam, Close-to-convexity and starlikeness of analytic functions, Tamkang J. Math. {\bf 46} (2015), no.~2, 111--119. 

\bibitem{mamin2}W. C. Ma\ and\ D. Minda, A unified treatment of some special classes of univalent functions, in {\it Proceedings of the Conference on Complex Analysis (Tianjin, 1992)}, 157--169, Conf. Proc. Lecture Notes Anal., I Int. Press, Cambridge, MA.

\bibitem{miller-dif-sub}S. S. Miller\ and\ P. T. Mocanu, Differential subordinations and inequalities in the complex plane, J. Differential Equations {\bf 67} (1987), no.~2, 199--211.

\bibitem{miller2} S. S. Miller\ and\ P. T. Mocanu, The theory and applications of second-order differential subordinations, Studia Univ. Babe\c s-Bolyai Math. {\bf 34} (1989), no.~4, 3--33. 

\bibitem{ural} S. S. Miller\ and\ P. T. Mocanu, {\it Differential Subordinations}, Monographs and Textbooks in Pure and Applied Mathematics, 225, Dekker, New York, 2000.

\bibitem{pap}E. Paprocki\ and\ J. Sok\'o\l, The extremal problems in some subclass of strongly starlike functions, Zeszyty Nauk. Politech. Rzeszowskiej Mat. No. 20 (1996), 89--94.

\bibitem{pot}Y. Polato\u glu\ and\ M. Bolcal, Some radius problem for certain families of analytic functions, Turkish J. Math. {\bf 24} (2000), no.~4, 401--412. 

\bibitem{darus}V. Ravichandran,\, M.\ Darus and\ N.\ Seenivasagan, On a criteria for strong starlikeness, The Australian J. Math. Anal. Appl. {\bf 2} (2005), no.~1 , Art.\ 6, 1--12.

    \bibitem{sharma}V. Ravichandran and K. Sharma, Sufficient conditions for starlikeness, J. Korean Math. Soc. {\bf 52} (2015), no.~4, 727–-749.

\bibitem{rob}M. S. Robertson, Certain classes of starlike functions, Michigan Math. J. {\bf 32} (1985), no.~2, 135--140.

\bibitem{Rus-Vsing}St. Ruscheweyh\ and\ V. Singh, On the order of starlikeness of hypergeometric functions, J. Math. Anal. Appl. {\bf 113} (1986), no.~1, 1--11.



\bibitem{shan}T. N. Shanmugam, Convolution and differential subordination, Internat. J. Math. Math. Sci. {\bf 12} (1989), no.~2, 333--340. 

\bibitem{vir} S. Sivaprasad Kumar, V. Kumar,  V. Ravichandran and N. E. Cho, Sufficient conditions for starlike  functions associated with the lemniscate of Bernoulli,  J. Inequal. Appl.  \textbf{2013} (2013) Art.\ 176, 13pp.


\bibitem{jain}K. Sharma, N. K. Jain and V. Ravichandran, Starlike functions associated with a cardioid, Afr. Mat. (Springer) {\bf 27} (2016), no.~5, 923--939.

\bibitem{sharma2}K. Sharma and V. Ravichandran, Applications of subordination theory to starlike functions, Bull. Iranian Math. Soc. {\bf 42} (2016), no.~3, 761--777.

 \bibitem{sharma1}K. Sharma and V. Ravichandran, Sufficient conditions for Janowski starlike functions, Stud. Univ. Babe\c s-Bolyai Math. {\bf 61} (2016), no.~1, 63--76.

\bibitem{sokol09b}  J. Sok\'o\l, Radius problems in the class $\mathcal{SL}$, Appl. Math. Comput. {\bf 214} (2009), no.~2, 569--573.

\bibitem{TM} N. M. Temme, {\it Special Functions}, Wiley, New York, 1996.

 \end{thebibliography}
\end{document}